\documentclass[11pt,a4paper]{article}

\usepackage{amsmath,amsthm,amsfonts}
\usepackage{url,cite}

\allowdisplaybreaks[3]

\usepackage[OT2,T1]{fontenc}
\usepackage[russian,english]{babel}

\newcommand\textcyr[1]{{\cyr #1}}

\usepackage{fouriernc}

\newcommand{\uhr}{\mathbin{\lceil}}

\renewcommand{\Tilde}{\widetilde}
\renewcommand{\Hat}{\widehat}

\newcommand{\RR}{\mathbb{R}}
\newcommand{\ZZ}{\mathbb{Z}}
\newcommand{\CC}{\mathbb{C}}
\newcommand{\NN}{\mathbb{N}}

\newcommand{\cH}{\mathcal{H}}
\newcommand{\cG}{\mathcal{G}}
\newcommand{\cB}{\mathcal{B}}
\newcommand{\cD}{\mathcal{D}}
\newcommand{\ELL}{\mathcal{L}}

\newtheorem{theorem}{Theorem}
\newtheorem*{theorem*}{Theorem}
\newtheorem{prop}[theorem]{Proposition}
\newtheorem{lemma}[theorem]{Lemma}
\newtheorem{corol}[theorem]{Corollary}

\theoremstyle{definition}
\newtheorem{defin}[theorem]{Definition}

\DeclareMathOperator{\supp}{supp}
\DeclareMathOperator{\spec}{spec}
\DeclareMathOperator{\dom}{dom}
\DeclareMathOperator{\ran}{ran}

\begin{document}

\title{\bf New Relations Between Discrete and~Continuous Transition Operators on~(Metric) Graphs}

\author{{\sc Daniel Lenz} \dag\quad and \quad {\sc Konstantin Pankrashkin} \ddag\\[\bigskipamount]
\dag{} \small Mathematisches Institut\\
\small Fakult\"at f\"ur Mathematik und Informatik\\
\small Friedrich-Schiller-Universit\"at Jena\\
\small Ernst-Abbe-Platz 2, 07743 Jena, Germany\\
\small Webpage: \url{http://www.analysis-lenz.uni-jena.de}\\
\small E-mail: \url{daniel.lenz@uni-jena.de}\\[\medskipamount]
\ddag{} \small Laboratoire de math\'ematiques, UMR 8628\\
\small Universit\'e Paris-Sud, B\^atiment 425\\
\small 91405 Orsay Cedex, France\\
\small Webpage: \url{http://www.math.u-psud.fr/~pankrash/}\\
\small E-mail: \url{konstantin.pankrashkin@math.u-psud.fr}
}

\date{}

\maketitle

\begin{abstract}
We establish several new relations between the discrete transition operator, the continuous Laplacian
and the averaging operator associated with combinatorial and metric graphs. It is shown that these operators
can be expressed through each other using explicit expressions. In particular, we show that
the averaging operator is closely related with the solutions of the associated wave equation.
The machinery used allows one to study a class of infinite graphs without assumption
on the local finiteness.
\end{abstract}


\section{Introduction}

Given a graph, one can define various associated operators, the discrete transition operator
and the continuous Laplacian being among the most prominent ones. A less known example is the continuous
transition operator (averaging operator) introduced in the papers \cite{CW,sw} which served as an immediate
motivation for the present work. In fact, it is known that the spectral properties 
of the three operators are closely related, see e.g.~\cite{CW,Cat}.
The aim of the present paper is to show that the relation is actually much more intimate,
and that the operators can be expressed through each other in a rather explicit way.
We present a machinery which allows one to bring the previous studies
to a common denominator and to work under very weak assumptions on the graph structure; in particular,
we are able to study a class of graphs which are not locally finite.
In addition, we establish a new link between the continuous transition operator
and the wave equation by adapting a d'Alembert-type formula we discovered in
the works~\cite{kop1,kop2} and which seems to be virtually unknown.
To be more precise and to present the results, let us introduce some notation;
we mostly follow the conventions used in the paper~\cite{CW}.

Let $X$ be a countable connected graph with symmetric neighborhood relation $\sim$
and without loops and multiple edges. It will be viewed as a one-complex by identifying each edge
with a copy of the interval $(0,1)$ and by gluing the edges together
at common endpoints; the symbol $X^0$ stands for the set of the vertices of $X$, and $X^1$ denotes its one-skeleton (i.e.
the associated \emph{metric graph}).
We denote by $(xy,t)$ the point of $X^1$ lying at distance $t\in[0,1]$ from $x$ on the non-oriented edge $[x,y]=[y,x]$, $x\sim y$.
One has the obvious identifications $(xy,0)\equiv x$ and $(xy,t)\equiv(yx,1-t)$.
In this way, the usual discrete graph metric
on the vertex set has a natural extension to $X^1$. 
Each edge $[x,y]$ will be equipped with a strictly positive conductance $c(xy)=c(yx)$, and
throughout the paper we assume that
\begin{equation}
        \label{eq-m0}
m^0(x):=\sum_{y:y\sim x} c(xy)<\infty  \qquad \text{ for any } x\in X^0.
\end{equation}
For a part of the results we will have to impose the stronger condition
\begin{equation}
     \label{eq-m1}
\sum_{y:y\sim x} \sqrt{c(xy)} <+\infty \qquad \text{ for any } x\in X^0,
\end{equation}
it will always be mentioned explicitly.
The function $m^0$ defines a discrete measure on $X^0$ which will be denoted by the same symbol; this gives rise
to the associated discrete Hilbert space $\ell^2(X^0,m^0)$ consisting of the functions
$h:X^0\mapsto \CC$ with
\[
\|h\|^2_{\ell^2(X^0,m^0)}=\sum_{x\in X^0}m^0(x)\big|h(x)\big|^2<\infty.
\]

On $X^1$ we introduce the continuous weighted Lebesgue measure $m^1$ which at
the point $(xy,t)$ is given by $c(xy)dt$ if $t\in(0,1)$, and the vertex set
has $m^1$-measure equal to zero. 
The pair $(X,c)$ together with the above measures
is usually called a \emph{network}.

For any function $F:X^1\to\CC$ and for $x,y\in X^0$ with $y\sim x$ denote by $F_{xy}$ the function
$t\mapsto F(xy,t)$, then the Hilbert space $\ELL^2(X^1,m^1)$ is exactly the space of all measurable
functions such that
\[
\|F\|^2_{\ELL^2(X^1,m^1)}=\dfrac{1}{2}\sum_{x\in X^0}\sum_{y:y\sim x}c(xy) \|F_{xy}\|^2_{\ELL^2(0,1)}<\infty;
\]
the factor $\frac 12$ is due to the fact that each edge appears twice in the sum.

Let us introduce several operators associated with a network.
The first one is the discrete transition operator $P$ acting on functions $g:X^0\to \CC$ by
\begin{equation}
          \label{eq-opp}
P g(x)=\dfrac{1}{m^0(x)}\sum_{y:y\sim x} c(xy)g(y);
\end{equation}
it is a bounded self-adjoint operator in $\ell^2(X^0,m^0)$
whose norm does not exceed $1$.
The second one is the continuous (positive) Laplace operator $L$ acting in the space
$\ELL^2(X^1,m^1)$ as the second derivative. More precisely,
for any function $F$ the prime sign will denote the derivation
with respect to the length parameter, i.e.
$F'(xy,t):=F'_{xy}(t)$.
For $k\in\NN$ introduce the space
\[
\Hat H^k(X^1,m^1):=\big\{
F\in \ELL^2(X^1,m^1):\, F^{(j)}\in \ELL^2(X^1,m^1),
\quad j=1,\dots,k \big\},
\]
where $F^{(j)}$ means the $j$th order derivative,
then the operator $L$ acts as the
$L F=-F''$ on functions
$F\in \Hat H^2(X^1,m^1)$
 satisfying
the boundary conditions
\begin{gather} 
              \label{eq-fcont}
F(xu,0)=F(xv,0)=: F(x) \text{ for all } x,u,v\in X^0 \text{ with } u,v\sim x,\\
              \label{eq-dder}
F'(x)=0 \text{ for all } x\in X^0,
\end{gather}
where we denote
\begin{equation}
              \label{eq-fder}
F'(x):=\sum_{y:y\sim x} c(xy)F'(xy,0+);
\end{equation}
it is an unbounded self-adjoint operator in $\ELL^2(X^1,m^1)$, see Theorem~\ref{th1} below.
Finally, the third operator $A$ is the averaging operator over unit balls
introduced in \cite{CW} in generalizing the notion of a transition operator \cite{sw};
it acts as
\begin{multline*}
AF(xy,t)=\dfrac{1}{m^0(x)}\sum_{u:u\sim x}
c(xu)\int_0^{1-t} F(xu,s)ds\\
+\dfrac{1}{m^0(y)}\sum_{v:v\sim y}
c(yv)\int_0^t F(yv,s)ds,
\end{multline*}
and it can be easily checked that $A$ is a  bounded self-adjoint operator in the space $\ELL^2(X^1,m^1)$.

We remark that the study of the operator $P$ is a classical topic of the spectral theory of combinatorial
graphs, cf. the monographs \cite{chung,cdv1,soardi,woess}
and the survey \cite{mw}. The operator $L$ and more general differential operators are actively studied
during the last decades, see the monographs \cite{fam,BK,DZ,pok,post}
and the collections of papers~\cite{kuc1,lum,snow,AGA}, sometimes under the name~\emph{quantum graphs}.

The three operators are closely related to each other.
It seems that the first relations between $P$ and $L$ on finite graphs and on some specific infinite graphs
were obtained in~\cite{vB,nic1,nic2,roth}, and the paper~\cite{Cat} was the first one which
studied this question for infinite graphs without any specific requirement for their structure,
and  the following result was shown under some technical assumptions:
\begin{prop}[Theorems 1 and 3 in \cite{Cat}]\label{prop1}
For $\lambda\notin\Sigma:=\big\{(\pi n)^2:n\in\NN\big\}$ and $\star\in\{\mathrm{p},\mathrm{c}\}$
there holds: $\lambda\in \spec_\star L$ if and only if $\cos\sqrt\lambda \in \spec_\star P$.
\end{prop}
The discrete set $\Sigma$ plays a special role, see subsection \ref{ssdir} below.
This relation was improved in various directions
for locally finite graphs and the uniform conductivity $c(xy)\equiv 1$. In particular, it was shown in \cite{BGP08} that 
the above relation also holds for the other spectral types, i.e. for $\star\in\{\mathrm{p},\mathrm{pp},\mathrm{disc},\mathrm{ess},\mathrm{ac},\mathrm{sc}\}$.
In \cite{KP11} it was shown that some functions of $L$ are unitarily equivalent to $P$, and in \cite{KP12}
a certain explicit form for the associated unitary transformation was obtained.
Similar relations for some other boundary conditions and slightly more general operators
were obtained e.g. in~\cite{vBM,exdual,KP06,KP11}, see also Section~3.6 in~\cite{BK}, Section~2.4 in~\cite{post}
and references therein. We also note that various relations between $P$ and $L$ and their generalizations
appear in the physics literature due to their importance for the theory of superconducting networks~\cite{alx}.

As for the operator $A$, the following result was obtained so far:
\begin{prop}[Theorems 1.3 and 1.5 in \cite{CW}]\label{prop2}
The spectrum of $A$ admits the following description:
\[
\spec A=\{0\}\bigcup \Big\{\dfrac{\sin\omega}{\omega}: \omega\in\RR\setminus\{0\}, \, \cos\omega\in\spec P
\Big\}\bigcup \Big( \,\{1\}\mathop{\cap}\spec P\Big).
\]
For the point spectrum we have
\[
\spec_\mathrm{p} A\setminus\{0\}=\Big\{\dfrac{\sin\omega}{\omega}: \omega\in\RR\setminus\{0\}, \, \cos\omega\in\spec_\mathrm{p} P
\Big\}
\bigcup \Big( \,\{1\}\mathop{\cap}\spec_\mathrm{p} P\Big),
\]
and $0$ belongs to $\spec_\mathrm{p} A$ unless $m^0(X^0)=\infty$ and $X$ is a tree with the property that after removal of any edge
at least of the connected components is recurrent \textup{(}we refer to subsection \ref{ssdir} for the exact definition of a recurrent network\textup{).}
\end{prop}

The aim of the present paper is to improve the above results concerning the relations between the three operators
and to establish a link between the operator $A$ and the wave equation
associated with $L$. This link was made possible using a d'Alembert-type
formula for the solutions of the wave equations on finite graphs we found in the papers \cite{kop1,kop2}.
These works seem to be not very known and hardly available, so let us briefly repeat here the main points.
Assume for the moment that the graph $X$ is finite.
For a function $F$ on $X^1$ define an extension 
$\RR\ni t\mapsto \Tilde F(xy,t)\in \CC$ of each component $F_{xy}$
 using the following recursion: 
\begin{align}
   \label{eq-tilde}
 \text{if }t&\in [0,1],& \Tilde F(xy,t)&:=F(xy,t),\\
 \text{if }t&>1, &\Tilde F(xy,t)&:= \dfrac{2}{m^0(y)} \sum_{v:v\sim y} c(yv) \Tilde F(yv,t-1) -\Tilde F(yx,t-1), \label{eq-tt1}\\[\bigskipamount]
 \text{if }t&<0, &\Tilde F(xy,t)&:= \dfrac{2}{m^0(x)} \sum_{u:u\sim x} c(ux) \Tilde F(ux,t+1) -\Tilde F(yx,t+1).  \label{eq-tt2}
\end{align}
It is easy to check that $\Tilde F(xy,t)= \Tilde F(yx,1-t)$ for all $x,y\in X^0$ with $x\sim y$ and all $t\in\RR$.
For $\tau\in\RR$ define a function $C(\tau)F$ on $X^1$ by
\begin{equation}
     \label{eq-ctau}
C(\tau)F (xy,t)=\dfrac{\Tilde F(xy,t+\tau)+\Tilde F(xy,t-\tau)}{2}.
\end{equation}
The operators $F\mapsto C(\tau) F$ will be called \emph{d'Alembert operators}.
The following result is proved in \cite[Section 1.4]{kop2}:
\begin{prop}\label{prop3a}
Let $X$ be finite and let $G\in \dom L$ with $G_{xy}\in C^2\big([0,1]\big)$,
then the function $F(\tau):=C(\tau)G$ solves the wave equation for
$L$ with the initial state $G$ and the zero initial velocity.
More precisely, for any $\tau\in\RR$ one has $F(\tau)\in\dom L$ with $F(\tau)_{xy}\in C^2\big([0,1]\big)$,
and for any $(xy,t)\in X^1$ we have
\begin{gather*}
\dfrac{d^2F(\tau;xy,t)}{d\tau^2}=\dfrac{d^2F(\tau;xy,t)}{dt^2},\\[\medskipamount]
F(0;xy,t)=G(xy,t), \quad 
\dfrac{dF(\tau;xy,t)}{d\tau}\Big|_{\tau=0}=0.
\end{gather*}
\end{prop}
It is one of our objectives to extend this construction to the $\ELL^2$ setting
on infinite networks (not necessarily localy finite),
and we will see that the operator $A$ is closely related to the d'Alembert operators.

Let us describe the main results of the paper.  In section \ref{sec-lp}
we improve the result of Proposition~\ref{prop1}.
Let $\cB(\RR)$ denote the set of the Borel subsets of $\RR$.
For a self-adjoint operator $T$ acting in a separable Hilbert space $\cH$ we denote
by $E_T:\cB(\RR)\to \ELL(\cH)$ the operator-valued spectral measure associated with $T$,
i.e. for any $\Omega\in \cB(\RR)$ we have $E_T(\Omega):=1_\Omega(T)$, where
$1_\Omega:\RR\to\RR$ is the indicator function of $\Omega$. For the same $\Omega$,
we denote by $T_\Omega$ the operator $T E_T(\Omega)$ viewed as a self-adjoint operator
in the Hilbert space $\ran E_T(\Omega)$ equipped with the induced scalar product.
We will use frequently the following functions:
\[
\Phi(z,t):=\begin{cases}
t, & z=0,\\
\dfrac{\sin(z t)}{z}, & z\ne 0,
\end{cases}
\quad\text{and}\quad
\Phi(z):=\Phi(z,1)=\begin{cases}
1, & z=0,\\
\dfrac{\sin z}{z}, & z\ne 0.
\end{cases}
\]

Our result on the relation between the operators $L$ and $P$ is as follows.

\begin{theorem}\label{th1}
The following assertions hold true.
\begin{itemize}
\item[(a)] The operator $L$ is self-adjoint and non-negative.

\item[(b)] Let $n\in \NN\cup\{0\}$. Denote
\begin{align*}
J&:=[0,\pi^2), & I&:=(-1,1],& \text{ if } n=0,\\
J&:=\big(\pi^2 n^2,\pi^2(n+1)^2\big),& I&:=(-1,1),& \text{ if } n\ge 1,
\end{align*}
and introduce a function $\kappa:[-1,1]\to\overline{\mathstrut J}$ by
\[
\kappa(t)=\begin{cases}
(\pi n+\arccos t)^2, & n \text { is even},\\
\big(\pi(n+1)-\arccos t\big)^2, & n \text { is odd};
\end{cases}
\]
in other words, the function $\kappa$ is the inverse of the function $J\ni \lambda\mapsto \cos\sqrt\lambda\in I$.
Then the operator $L_J$ is unitarily equivalent to the operator $\kappa(P_I)$.

\item[(c)] Moreover, one has the representation
$L_J=U_J \kappa(P_I) U_J^*$, where $U_J:\ran E_P\big(I)\to\ran E_L(J)$ is the unitary operator given by
\[
U_J:=\sqrt 2\int_J \gamma(\lambda)\, d E_P(\cos\sqrt \lambda)
\]
with the operators $\gamma(z):\ell^2(X^0,m^0)\to\ELL^2(X^1,m^1)$ given by
\[
\big(\gamma(z) h\big)(xy,t)=
h(x)\dfrac{\Phi(\sqrt z,1-t)}{\Phi(\sqrt z,1)}+h(y)\dfrac{\Phi(\sqrt z,t)}{\Phi(\sqrt z,1)},
\]
and the adjoint $U^*_J:\ran E_L\big(J)\to\ran E_P(I)$ acts as
\[
U^*_J:=\sqrt 2\int_J d E_P(\cos\sqrt \lambda)\, \gamma(\lambda)^*;
\]
here the both integrals are understood as improper Riemann-Stieltjes integrals.
\item[(d)] In particular, let $\star\in\{\mathrm{p},\mathrm{pp},\mathrm{disc},\mathrm{ess},\mathrm{ac},\mathrm{sc}\}$
and $\lambda\notin\Sigma:=\big\{(\pi n)^2:n\in\NN \big\}$, then the condition
$\lambda\in \spec_\star L$ is equivalent to the condition $\cos\sqrt\lambda\in \spec_\star P$,
and there holds $\ker(L-\lambda)=\gamma(\lambda)\ker (P-\cos\sqrt\lambda)$.
\end{itemize}
\end{theorem}
We refer to the paper~\cite{KP12} for the exact definition of the operator-valued
integrals appearing in item~(c).
The proof of theprem~\ref{th1} is given in subsection \ref{ss-spl} and is based on a certain result on the spectra of self-adjoint extensions
obtained in \cite{KP11,KP12}; the abstract machinery is reviewed in subsection \ref{ss-bt}.
In Section \ref{sec3} we establish a Fourier-type expansion associated with $L$;
at this stage we have to impose the condition \eqref{eq-m1}. 
In Section~\ref{sec4} we switch to the study of the operators $C(\tau)$ given by \eqref{eq-ctau}.
We establish the following relation:

\begin{theorem}\label{th4}
Assume that the network $(X,c)$ satisfies \eqref{eq-m1}, then 
for any $\tau\in\RR$ the d'Alembert operator $C(\tau)$ and the continuous Laplacian
$L$ are related by $C(\tau)=\cos(\tau\sqrt L)$.
\end{theorem}

This relation is finally used to obtain the following representation for the operator $A$.
\begin{theorem}\label{th6}
Assume that the network $(X,c)$ satisfies the condition~\eqref{eq-m1}, then the averaging operator $A$ and the continuous Laplacian $L$ are related by
$A=\Phi(\sqrt L)$. 
\end{theorem}

In other words, we have a very simple interpretation of the operator $A$:
if $F\in\dom\sqrt L$, then $AF:=G(1)$, where $G$ is the solution
of the wave equation $G''+L G=0$ with $G(0)=0$ and $G'(0)=F$ (see subsection~\ref{swave}
for a more detailed discussion of the abstract wave equation).
As a consequence of Theorem~\ref{th6} and of the functional calculus of self-adjoint operators one obtains
\begin{corol}\label{corol1}
Assume that the condition~\eqref{eq-m1} holds, then for any index $\star\in\{\mathrm{p},\mathrm{pp},\mathrm{disc},\mathrm{ess},\mathrm{ac},\mathrm{sc}\}$ there holds
\[
\spec_\star A=\Big\{
\Phi(\sqrt\lambda): \lambda\in \spec_\star L
\Big\},
\]
and for any $\lambda\in\spec_\mathrm{p} L$ one has
\[
\ker \Big(A-\Phi(\sqrt\lambda)\Big)=\bigcup_{\mu\ge 0: \,\Phi(\sqrt{\mathstrut\mu})=\Phi(\sqrt{\mathstrut\lambda})}\ker(L-\mu).
\]
\end{corol}
Combining Corollary~\ref{corol1} with Theorem~\ref{th1} and with Corollary~\ref{corol13} below and using the fact
that $\Sigma=\Phi^{-1}(0)$ we obtain, in particular, a more detailed version of~Proposition~\ref{prop2}:
\begin{corol}
Assume that the condition~\eqref{eq-m1} is satisfied, then 
for any index $\star\in\{\mathrm{p},\mathrm{pp},\mathrm{disc},\mathrm{ess},\mathrm{ac},\mathrm{sc}\}$ there holds
\[
\spec_\star A\setminus\{0\}=\Big\{\Phi(\omega): \omega\in\RR\setminus\{\pi n:\, n\in\NN\} \text{ and } \cos\omega\in\spec_\star P
\Big\},
\]
and  for any $\lambda\in \RR\setminus\{0\}$ we have
\[
\ker(A-\lambda)=\bigcup_{\mu\ge 0:\, \Phi(\sqrt{\mathstrut \mu})=\lambda}\gamma(\mu)\ker(P-\cos\sqrt\mu)
\]
with the maps $\gamma$ defined in~Theorem~\ref{th1}(c).
Furthermore, there holds
\[
\ker A= \bigcup_{n=1}^\infty \ker(L-\pi^2 n^2),
\]
and the structure of each kernel at the right-hand side is described in Theorem~\ref{propker}.
In particular, the point $0$ belongs to $\spec_\mathrm{p} A$ unless $m^0(X^0)=\infty$ and
$X$ is a tree with the property that after removal of any edge
at least of the connected components is recurrent.
\end{corol}

One has a natural question whether the assumption \eqref{eq-m1} is really needed for the above results.
It seems to be a technical one, but we make use of it to establish
some local regularity properties of a Fourier-type expansion for $L$ which are of importance
for the subsequent constructions.
Note that the consideration of metric graphs which are not locally finite
already presents a certain challenge, and they only started to be studied in rather recent works~\cite{LSV,SSVW,schub}.
Fourier-type expansions on metric graphs were obtained so far for the locally finite case only,
see e.g. \cite{bls,LSS}, and our work can be also be viewed as a contribution in this direction.
Regardless of the preceding remarks, we believe that the assumption \eqref{eq-m1} as sufficiently general.
In particular, the both assumptions
\eqref{eq-m0} and \eqref{eq-m1} are valid for locally finite graphs,
and for the uniform case $c(xy)\equiv 1$ they both are equivalent to the local finiteness.
All our results are new even for the locally finite case, and Theorems \ref{th4} and~\ref{th6}
together with their corollaries are new even for the uniform case.

At last we remark that the methods we employ are very different
from those used to obtain the earlier results presented
in~Propositions~\ref{prop1}, \ref{prop2} and \ref{prop3a}, and our proofs
are completely independent; we hope that the machinery 
that we develop in the present work can be useful to study a larger class of~problems.

\section{Relations between $L$ and $P$}\label{sec-lp}

\subsection{Boundary triples}\label{ss-bt}

Let us recall some tools from the theory of self-adjoint extensions of symmetric operators.
We refer to \cite[Section 1]{BGP08} for a compact presentation
or to \cite{DM,GG} for a more detailed discussion.

\begin{defin}
Let $S$ be a closed linear operator in a separable
Hilbert space $\cH$ with the domain $\dom S$ which is dense in $\cH$.
Assume that there exist
an auxiliary Hilbert space $\cG$ and two linear maps $\Gamma,\Gamma':\dom S\to \cG$ such that
\begin{itemize}
\item for any $f,g\in\dom S$ there holds
\begin{equation}
       \label{eq-green}
\langle f,Sg\rangle_\cH-\langle Sf,g\rangle_\cH=\langle  \Gamma f,\Gamma' g\rangle_\cG-\langle\Gamma'f,\Gamma g\rangle_\cG,
\end{equation}
\item the map $(\Gamma,\Gamma'):\dom S\ni f\mapsto (\Gamma f,\Gamma'f)\in\cG\times\cG$ is surjective,
\item the set $\ker\,(\Gamma,\Gamma')$ is dense in $\cH$.
\end{itemize}
A triple $(\cG,\Gamma, \Gamma')$ with the above properties
is called a \emph{boundary triple} for the operator $S$. \qed
\end{defin}

It is known that an operator $S$ having a boundary triple $(\cG,\Gamma,\Gamma')$
has self-adjoint restrictions if and only if $S=S_0^*$, where
$S_0:=S\uhr\ker(\Gamma,\Gamma')$, see e.g. \cite[Theorem 1.13]{BGP08}; here and below the writing $T\uhr \cD$ means the restriction
of an operator $T$ to a subspace $\cD$.
From now on assume that this condition holds, then
the operators $H^0:=S\uhr \ker\Gamma$ and $H:=S\uhr\ker\Gamma'$
are self-adjoint \cite[Theorem 1.12]{BGP08}.
Let us discuss an approach to the spectral analysis of these operators.

One can show that for any $z\notin\spec H^0$ and any $g\in\cG$
one can find a unique $f\in\ker (S-z)$ with $\Gamma f=g$.
The map $\gamma(z): g\mapsto f$ is a linear map from $\cG$ to $\cH$ and an isomorphism between
$\cG$ and $\ker(S-z)$, see Theorem 1.23 in~\cite{BGP08}; it is sometimes referred to as \emph{Krein $\gamma$-field}
associated with the boundary triple. For the same $z$ denote by $M(z)$
the linear operator on $\cG$ acting as $M(z):=\Gamma' \gamma(z)$; this operator
will be referred to as the \emph{Weyl function}. The Weyl function encodes
an essential part of the information of the spectra of self-adjoint extensions,
see e.g. the review in~\cite{ABMN05}. The following theorem collects several known results.
\begin{theorem}
        \label{thmkp}
        Assume that the Weyl function $M$ has the form
$M(z)=n(z)^{-1}\big(m(z)-T\big)$, 
where $T$ is a bounded self-adjoint operator in $\cG$ and
$m$ and $n$ are scalar functions which are holomorphic outside $\spec H^0$.
Moreover, assume that there exists a spectral gap $J:=(a_0,b_0)\subset \RR\setminus \spec H^0$
such that $m$ and $n$ admit a holomorphic continuation to $J$, are both real-valued in $J$,
that $n\ne 0$ in $J$, and that $m(J)\cap \spec T\ne\emptyset$, then the following holds
\begin{itemize}
  \item[(a)] There exists an interval $K$ containing $m^{-1}(\spec T)\cap J$
  such that $m:K\to m(K)$ is a bijection; denote by $\mu$ the inverse function.
  \item[(b)] The operator $H_J$ is unitarily equivalent to $\mu(T_{m(J)})$. 
  \item[(c)] Moreover, there holds $H_J=U_J \mu(T_{m(J)}) U_J^*$, where
  $U_J$ is a unitary operator from $\ran E_T\big(m(J)\big)$ to $\ran E_H (J)$ given by 
  \[
U_J = \int_J \sqrt{\dfrac{n(\lambda)}{m'(\lambda)}}\, \gamma(\lambda) dE_T\big(m(\lambda)\big),
  \]
  and its adjoint $U^*_J:\ran E_H (J)\to \ran E_T\big(m(J)\big)$ acts as
  \[
U^*_J = \int_J \sqrt{\dfrac{n(\lambda)}{m'(\lambda)}} dE_T\big(m(\lambda)\big) \, \gamma(\lambda)^*,
  \]
   and both integrals are understood as improper Riemann-Stieltjes integrals.
  \item[(d)] For  any $\lambda\in \RR\setminus\spec H^0$ and
  $\star\in\{\mathrm{p},\mathrm{pp},\mathrm{disc},\mathrm{ess},\mathrm{ac},\mathrm{sc}\}$
  the condition $\lambda\in\spec_\star L$
  is equivalent to the condition $m(\lambda)\in\spec_\star T$, 
  and one has the equality $\ker(L-\lambda)=\gamma(\lambda)\ker \big(T-m(\lambda)\big)$.
\end{itemize}
\end{theorem}
The parts (a) and (b) were obtained in \cite[Theorem 2]{KP11} and the part (c) was shown in \cite[Corollary 11]{KP12}.
The part (d) was actually proved in the earlier paper \cite[Theorem 3.16]{BGP08}, but it can be viewed as a direct corollary of~(b) and/or~(c).

\subsection{Spectral analysis of $L$}\label{ss-spl}

Let us put the study of the operator $L$ into the framework of boundary triples.
Recall the Sobolev inequalities: there exist $a,b>0$ such that
for all $f\in H^2(0,1)$ we have
\begin{equation}
           \label{eq-sob}
\|f\|^2_{\ELL^\infty(0,1)}+\|f'\|^2_{\ELL^\infty(0,1)}\le a\|f''\|^2_{\ELL^2(0,1)}+b\|f\|^2_{\ELL^2(0,1)}
\end{equation}
and for all $f\in H^1(0,1)$ there holds
\begin{equation}
           \label{eq-sob2}
\|f\|^2_{\ELL^\infty(0,1)}\le a\|f'\|^2_{\ELL^2(0,1)}+b\|f\|^2_{\ELL^2(0,1)}.
\end{equation}
Introduce an operator $S$ acting in $\ELL^2(X^1,m^1)$ as $SF=-F''$
on the domain
$\dom S:=\big\{
F\in \Hat H^2(X^1,m^1): \text{Eq. \eqref{eq-fcont} holds}
\big\}$.

\begin{lemma}\label{lem4}
The operator $S$ is closed.
\end{lemma}

\begin{proof}
Consider the operator $\Tilde S$ acting on the domain $\dom \Tilde S=\Hat H^2(X^1,m^1)$
as $\Tilde S F= -F''$. This operator is obviously closed being an orthogonal
direct sum of closed operators.
Furthermore, for all $x\in X^0$ and $u,v\sim x$ consider the functionals
$l_{xuv}$ on $\dom \Tilde S$ given by $l_{xuv}F=F(xv,0)-F(xu,0)$. By \eqref{eq-sob},
all these functionals are continuous in the graph norm of $\Tilde S$, hence
their kernels are closed in the graph norm of $\Tilde S$.
As $S$ is the restriction of $\Tilde S$
to the intersection of these kernels, it is a closed operator.
\end{proof}

For $F\in\dom S$ define functions $\Gamma F:X^0\to \CC$ and $\Gamma'F:X^0\to \CC$ by
\[
(\Gamma F)(x):= F(x),
\quad
(\Gamma' F)(x):=\dfrac{F'(x)}{m^0(x)}, \quad x\in X^0;
\]
here we use the notation of Eqs.~\eqref{eq-fcont} and \eqref{eq-fder}.
The following lemma adapts Lemma~2 from~\cite{KP06} to the case
of weighted networks.
\begin{lemma}\label{lem5}
The operators $\Gamma,\Gamma'$ map $\dom S$ to $\ell^2(X^0,m^0)$, and
the triple $\big(\ell^2(X^0,m^0),\Gamma,\Gamma'\big)$
is a boundary triple for $S$. 
\end{lemma}

\begin{proof}
Let us show first that for any $F\in \dom S$ one has
$\Gamma F\in \ell^2(X^0,m^0)$ and $\Gamma'F\in \ell^2(X^0,m^0)$.
Every component $F_{xy}$ belongs to $H^2(0,1)$, and using \eqref{eq-sob} we estimate
\begin{multline*}
\sum_{x\in X^0} m^0(x)\big|(\Gamma F)(x)\big|^2=
\sum_{x\in X^0} m^0(x)\big|F(x)\big|^2\\
\begin{aligned}
=\,&\sum_{x\in X^0} \sum_{y:y\sim x} c(xy)\big|F(x)\big|^2=\sum_{x\in X^0} \sum_{y:y\sim x} c(xy)\big|F_{xy}(0)\big|^2\\
\le\,& a\sum_{x\in X^0} \sum_{y:y\sim x} c(xy)\big\|F''_{xy}\big\|^2_{\ELL^2(0,1)}
+b\sum_{x\in X^0} \sum_{y:y\sim x} c(xy)\big\|F_{xy}\big\|^2_{\ELL^2(0,1)}\\
=\,&2a \|F''\|^2_{\ELL^2(X^1,m^1)}+2b\|F\|^2_{\ELL^2(X^1,m^1)}<\infty.
\end{aligned}
\end{multline*}
Hence, $\Gamma F\in \ELL^2(X^0,m^0)$.
Furthermore, using the Cauchy-Schwartz inequality one can estimate
\begin{align*}
\Big|\sum_{y:y\sim x} c(xy)F'(xy,0+)
\Big|^2&=\Big|\sum_{y:y\sim x} \sqrt{c(xy)}\cdot \sqrt{c(xy)}F'(xy,0+)
\Big|^2\\
&\le
\Big(
\sum_{y:y\sim x} c(xy)
\Big)\cdot
\Big(\sum_{y:y\sim x} c(xy)\big|F'(xy,0+)\big|^2
\Big)\\
&=m^0(x)\cdot \sum_{y:y\sim x} c(xy)\big|F'(xy,0+)\big|^2,
\end{align*}
hence
\begin{multline*}
\sum_{x\in X^0} m^0(x)\big|(\Gamma' F)(x)\big|^2= 
\sum_{x\in X^0} \dfrac{\big|F'(x)\big|^2}{m^0(x)}\\
\begin{aligned}
&=
\sum_{x\in X^0} \dfrac{1}{m^0(x)}\cdot \Big|\sum_{y:y\sim x} c(xy)F'(xy,0+)
\Big|^2\\
&\le
\sum_{x\in X^0}\sum_{y:y\sim x} c(xy)\big|F'(xy,0+)\big|^2
=
\sum_{x\in X^0}\sum_{y:y\sim x} c(xy)\big|F'_{xy}(0)\big|^2\\
&\le a\sum_{x\in X^0} \sum_{y:y\sim x} c(xy)\big\|F''_{xy}\big\|^2_{\ELL^2(0,1)}
+b\sum_{x\in X^0} \sum_{y:y\sim x} c(xy)\big\|F_{xy}\big\|^2_{\ELL^2(0,1)}\\
&=2a \|F''\|^2_{\ELL^2(X^1,m^1)}+2b\|F\|^2_{\ELL^2(X^1,m^1)}<\infty,
\end{aligned}
\end{multline*}
which implies $\Gamma'F\in \ell^2(X^0,m^0)$.

Now let us show that the map
\[
\dom S\ni F\mapsto(\Gamma F,\Gamma' F)\in \ell^2(X^0,m^0)\times \ell^2(X^0,m^0)
\]
is surjective.
Let us pick  functions $f_{jk}\in H^2[0,1]$ such that
$f_{jk}^{(i)}(l)=\delta_{ij}\delta_{kl}$ for all $i,j,k,l\in\{0,1\}$ and denote
\[
M_0:=\max_{j,k\in\{0,1\}}\Big( \|f_{jk}\|^2_{\ELL^2(0,1)}+\|f''_{jk}\|^2_{\ELL^2(0,1)}\Big).
\] 
Now let $h,h'\in \ell^2(X^0,m^0)$. Define a function 
$F:X^1\to\CC$ by
\[
F_{xy}=h(x)f_{00}+h(y)f_{01}+h'(x)f_{10}-h'(y)f_{11}.
\]
One has $F_{xy}\in H^2(0,1)$
and
\[
\|F_{xy}\|^2_{\ELL^2(0,1)}+\|F''_{xy}\|^2_{\ELL^2(0,1)}\le 4M_0^2\Big(
\big|h(x)\big|^2+\big|h(y)\big|^2
+
\big|h'(x)\big|^2+\big|h'(y)\big|^2
\Big)
\]
hence
\begin{multline*}
\|F\|^2_{\ELL^2(X^1,m^1)}+\|F''\|^2_{\ELL^2(X^1,m^1)}\\
\begin{aligned}
=\, &
\dfrac12\,\sum_{x\in X^0}\sum_{y: y\sim x}c(xy)\cdot\big(\|F_{xy}\|^2_{\ELL^2(0,1)}+ \|F''_{xy}\|^2_{\ELL^2(0,1)}\big)\\
\le\, &
2M_0^2\,\sum_{x\in X^0}\sum_{y: y\sim x}c(xy)\Big(\big|h(x)\big|^2+\big|h(y)\big|^2\big)\\
&+
2M_0^2 \,\sum_{x\in X^0}\sum_{y: y\sim x}c(xy)\Big(\big|h'(x)\big|^2+\big|h'(y)\big|^2\big)\\
=\, &4M_0^2\,\sum_{x\in X^0}\sum_{y: y\sim x}c(xy)\big|h(x)\big|^2+
4M_0^2 \,\sum_{x\in X^0}\sum_{y: y\sim x}c(xy)\big|h'(x)\big|^2\\
=\,&8M_0^2 \|h\|_{\ell^2(X^0,m^0)}^2+8M_0^2 \|h'\|_{\ell^2(X^0,m^0)}^2<\infty,
\end{aligned}
\end{multline*}
which shows that $F\in\Hat H^2(X^1,m^1)$.
Furthermore, for any $x\in X^0$ and $y\sim x$
one has $F(xy,0)=F_{xy}(0)=h(x)$, hence $F\in\dom S$ and $\Gamma F=h$.
Now we have $F'_{xy}(0)=h'(x)$, hence
\[
\Gamma' F(x)=\dfrac{1}{m^0(x)}\sum_{y\sim x} c(xy) F'_{xy}(0)=
\dfrac{1}{m^0(x)}\sum_{y:y\sim x} c(xy) h'(x)=h'(x),
\]
i.e. $\Gamma'F=h'$. As $h,h'$ are arbitrary, the surjectivity is proved.
As the subspace $\ker(\Gamma,\Gamma')$ contains all functions $F\in \ELL^2(X^1,m^1)$
with $F_{xy}\in C_c^\infty(0,1)$, it is dense in $\ELL^2(X^1,m^1)$.
It remains to show the validity of the identity \eqref{eq-green}. We have, for any
$F,G\in \dom S$,
\begin{multline*}
\langle F,SG\rangle_{\ELL^2(X^1,m^1)}-\langle SF,G\rangle_{\ELL^2(X^1,m^1)}
\\
\begin{aligned}
=\,&\dfrac{1}{2}\,\sum_{x\in X^0}
\sum_{y:y\sim x} c(xy)\Big[
\int_0^1F''_{xy}(t) \overline{G_{xy}(t)}dt-\int_0^1F_{xy}(t)\overline{G''_{xy}(t)}dt
\Big]\\
=\, & \dfrac{1}{2}\,\sum_{x\in X^0}
\sum_{y:y\sim x} c(xy)\Big[
F_{xy}(0)\overline{G'_{xy}(0)}-F'_{xy}(0)\overline{G_{xy}(0)}\\
&+
F'_{xy}(1)\overline{G'_{xy}(1)}-F_{xy}(1)\overline{G'_{xy}(1)}
\Big]\\
=\, &\sum_{x\in X^0}\sum_{y:y\sim x}c(xy)\Big[
F_{xy}(0)\overline{G'_{xy}(0)}-F'_{xy}(0)\overline{G_{xy}(0)}\Big]\\
=\, &\sum_{x\in X^0}F(x)\overline{\sum_{y:y\sim x} c(xy)G'_{xy}(0)}
- \sum_{x\in X^0}\sum_{y:y\sim x} c(xy)F'_{xy}(0) \overline{G(x)}\\
=\, & \sum_{x\in X^0}F(x) \overline{G'(x)} - \sum_{x\in X^0}F'(x)\overline{G(x)}\\
=\, &\big\langle \Gamma F,\Gamma'G\big\rangle_{\ell^2(X^0,m^0)}-\big\langle \Gamma' F,\Gamma G\big\rangle_{\ell^2(X^0,m^0)},
\end{aligned}
\end{multline*}
which concludes the proof.
\end{proof}

Now we have all necessary tools to prove Theorem~\ref{th1}.

\begin{proof}[Proof of Theorem \ref{th1}]
The  restriction
$H^0:=S\uhr \ker\Gamma$ is just the direct sum of the second derivative operators
with the Dirichlet boundary conditions on all edges and is thus self-adjoint.
Therefore, $S$ has a self-adjoint restriction.
As $L=S\uhr\ker\Gamma'$,  the operator $L$ is self-adjoint, and it is non-negative by direct computation (see e.g. Lemma \ref{lem6} below),
which proves (a).

The spectrum of $H^0$ coincides with the Dirichlet spectrum of the segment $[0,1]$, which is exactly
the set $\Sigma:=\big\{(\pi n)^2:n\in\NN\big\}$.
Let us now now calculate the $\gamma$-field $\gamma(z)$ and the Weyl function $M(z)$ associated with $S$
and the boundary triple constructed above.
Note that for all $\lambda\in\CC$ we have
$\Phi'(\lambda,x)\equiv d\Phi(\lambda,x)/dx=\cos (\lambda x)$.
Now let $h\in \ell^2(X^0,m^0)$ and $z\notin\Sigma$.
Consider the function $F:X^1\to\CC$ given by
\begin{equation}
     \label{eq-fxy}
F_{xy}(t)=h(x)\dfrac{\Phi(\sqrt z,1-t)}{\Phi(\sqrt z,1)}+h(y)\dfrac{\Phi(\sqrt z,t)}{\Phi(\sqrt z,1)}.
\end{equation}
By direct computation, which is almost identical to the one done in the proof of Lemma \ref{lem5},
we have $F\in \Hat H^2(X^1,m^1)$, $\Gamma F=h$
and $-F''_{xy}=zF_{xy}$ for all $x\sim y$. In other words,
$F$ belongs to $\dom S$ and $SF=zF$, hence $F=\gamma(z)h$.
Applying  the map $\Gamma'$ we obtain:
\begin{align*}
M(z)h(x)&=\Gamma'\gamma(z)h(x)=(\Gamma'F)(x)=\dfrac{1}{m^0(x)}
\sum_{y:y\sim x} c(xy)F'_{xy}(0)\\
&=\dfrac{1}{m^0(x) \Phi(\sqrt z,1)}\sum_{y:y\in x} c(xy)\Big( - h(x) \Phi'(\sqrt z,1)+h(y) \Phi'(\sqrt z,0)\Big)\\
&=\dfrac{1}{m^0(x) \Phi(\sqrt z,1)}\Big(
-m^0(x)\cos\sqrt z\,h(x)+ \dfrac{1}{m^0(x)}\sum_{y\in x} c(xy) h(y)
\Big)\\
&=\dfrac{1}{m^0(x) \Phi(\sqrt z,1)}\big(
-m^0(x)\cos\sqrt z\,h(x)+ P h(x)\big).
\end{align*}
The function $z\mapsto \Phi(\sqrt z,1)$ is holomorphic and does not vanish on $\RR\setminus\Sigma$, and we arrive at
\[
M(z)=\dfrac{P -\cos\sqrt z}{\Phi(\sqrt z,1)}.
\]
Now we remark that the function $\kappa$
is exactly the inverse of $J\ni \lambda\mapsto \cos\sqrt\lambda \in I$ and that $\kappa(J)=I$.
Therefore, the assertions (b), (c) and (d) follow from Theorem \ref{thmkp}.
\end{proof}

\subsection{The Dirichlet eigenvalues}\label{ssdir}

Theorem \ref{th1} gives the complete description of the spectrum of $L$ in terms of $P$
outside the discrete set $\big\{(\pi n)^2: n\in\NN\big\}$. To complete the spectral analysis of $L$
it remains to study the kernels $\ker(L-\pi^2n^2)$
for $n\in\NN$. Actually this problem was completely solved in \cite[Section 3]{Cat}, and
we include this subsection for the sake of completeness.

Let us introduce first some notation and recall some notions from the graph theory.
Let $E$ be the set of the \emph{ordered} pairs $xy$ with $x,y\in X^0$ and $x\sim y$; by
$\ell^2(E,c)$ we denote the Hilbert space of the functions $a:E\to\CC$ with the norm
\[
\|a\|^2_{\ell^2(E,c)}:=\sum_{xy\in E} c(xy)\big|a(xy)\big|^2<\infty.
\]

As usually, for $x\in X^0$ denote by $\delta_x$ the function on $X^0$
which is equal to one at $x$ and is equal to zero at all other points.
The network $(X,c)$ is called \emph{transient} if
\[
\sum_{n=0}^\infty \langle P^n\delta_x,\delta_y\rangle_{\ell^2(X^0,m^0)}<\infty
\text{ for some (and then for any) } x,y\in X^0;
\]
otherwise it is called \emph{recurrent}.
One may consult \cite{soardi,woess} for a further discussion of these notions.
The following results are obtained in Theorem~1 \cite{Cat}:
\begin{theorem}\label{propker}
For any $n\in\NN$ one has
$\ker(L-\pi^2n^2)= C_n (\Lambda_n) \oplus S_n (\Pi_n)$,
where
\begin{align*}
\Lambda_n:=& \big\{h\in \ell^2(X^0,m^0):\, h(x)=(-1)^n h(y) \text{ for }x\sim y \big\},\\
\Pi_n:=& \Big\{a\in \ell^2(E,c):\, a(yx)=(-1)^{n+1} a(xy) \text{ for all } xy\in E\\
&\quad 
\text{ and } \sum_{y:y\sim x}c(xy)a(xy)=0 \text{ for all } x\in X^0\Big\},
\end{align*}
and the maps $C_n:\Lambda_n\to \ELL^2(X^1,m^1)$ and $S_n:\Pi_n\to \ELL^2(X^1,m^1)$ act as
\[
(C_n h)(xy,t):= h(x)\cos(\pi n t),\quad
(S_n a)(xy,t):=a(xy)\sin(\pi nt).
\]
Moreover,
\[
\dim\Lambda_n=\begin{cases}
1, & m^0(X^0)<\infty, \quad n \text{ is odd and $X$ is bipartite},\\
1, & m^0(X^0)<\infty,\quad n \text{ is even}, \\
0, & \text{in all other cases.}
\end{cases}
\]
and $\dim \Pi_n=0$ if and only if one of the following two conditions holds:
\begin{itemize}
\item $X$ is a tree with the property that after removal of any edge
at least of the connected components is recurrent,
\item $n$ is odd, $X$ has only one cycle which is odd, and $(X,c)$ is recurrent.
\end{itemize}
\end{theorem}
An elementary analysis leads to the following observation:
\begin{corol}\label{corol13} The following two conditions are equivalent:
\begin{itemize}
\item $\bigcup_{n=1}^\infty \ker(L-\pi^2 n^2)=\{0\}$,
\item $m^0(X^0)<\infty$ and
$X$ is a tree with the property that after removal of any edge
at least of the connected components is recurrent.
\end{itemize}
\end{corol}

\section{Fourier-type expansions associated with $L$}\label{sec3}

We recall first some basic notions concerning Fourier-type expansions associated with self-adjoint
operators following \cite{PGW}. Let $\cH$ be a separable Hilbert space with the scalar product $\langle \cdot,\cdot\rangle$
and $T\ge 1$ be a self-adjoint operator $\cH$. Define two auxiliary Hilbert spaces $\cH_\pm=\cH_\pm(T)$ as follows:
by $\cH_+$ we denote the domain $\dom T$ equipped with the scalar product $\langle f,g\rangle_+=\langle T f,Tg\rangle$,
and $\cH_-$ will be the completion of $\cH$ with respect to the scalar product
$\langle f,g\rangle_-:=\langle T^{-1}f,T^{-1}g\rangle$. The scalar product on $\cH$
can be then naturally extended to a sesquilinear map $\langle\cdot, \cdot\rangle:\cH_+\times \cH_-\to\CC$.

\begin{defin}\label{defin2}
Let $N$ be a positive integer or infinity, $H$ be a self-adjoint operator in $\cH$
and $\mu$ be a spectral measure for $H$. A sequence of subsets $M_j\subset\RR$, $j=1,\dots,N$,
with $M_j\supset M_{j+1}$ together with a unitary map
\[
U=(U_j):\cH\to \bigoplus_{j=1}^N L^2(M_j,d\mu)
\]
is said to be an \emph{ordered spectral representation} of $H$ if $U \varphi(H)=M_\varphi U$
for every measurable function $\varphi$ on $\RR$. Here and below $M_\varphi$ denotes the operator of multiplication
by $\varphi$ in $L^2(M,d\mu)$. 

We say that measurable functions
\[
\varphi_j:M_j\to\cH_-, \quad M_j\ni\lambda\mapsto \varphi_{j,\lambda}\in\cH_-, \quad
j=1,\dots,N,
\]
realize a \emph{Fourier-type expansion} for $H$ if the following conditions are satisfied:
\begin{enumerate}
\item[(a)] $ U_j f(\lambda)=\big\langle f,\varphi_{j,\lambda}\big\rangle$ for $\mu$-a.e. $\lambda$ and for all $f\in\cH_+$,
\item[(b)] For every $g=(g_j)\in\bigoplus_{j=1}^N L^2(M_j,d\mu)$ there holds
\[
Ug=\lim_{\substack{n\to N\\E\to\infty}} \sum_{j=1}^n \int_{M_j\cap(-E,E)} g_j(\lambda) \varphi_{j,\lambda}\, d\mu(\lambda).
\]
\item[(c)] For every $f\in\cH$ there holds
\[
f=\lim_{\substack{n\to N\\E\to\infty}} \sum_{j=1}^n \int_{M_j\cap(-E,E)} (U_j f)(\lambda) \varphi_{j,\lambda}\, d\mu(\lambda),
\]
\item[(d)] for any vector $f\in \big\{ g\in \cH_+\mathop{\cap}\dom H: Hg\in \cH_+\big\}$ one has
$\langle Hf,\varphi_{j,\lambda}\rangle=\lambda \langle f, \varphi_{j,\lambda}\rangle$. \qed
\end{enumerate}
\end{defin}
We note that an ordered spectral representation always exists and is unique up to renumbering,
see e.g.~\cite[Chapter~8]{weid},
and Fourier-type expansions can be constructed using the following result:
\begin{prop}[Theorem 1 in \cite{PGW}]\label{thm-gf}
Assume that there exists a continuous bounded function $\gamma:\RR\to \CC$
with $|\gamma|>0$ on $\spec H$ such that $\gamma(H) T^{-1}$ is a Hilbert-Schmidt operator.
Then there exist measurable functions $\varphi_j:M_j\to\cH_-$, $\lambda\mapsto \varphi_{j,\lambda}$,
$j=1,\dots,N$, realizing a Fourier-type expansion for $H$.
\end{prop}

Let as apply the above machinery to the operator $L$. From now on let $\mu$ be its spectral measure
and $(M_j,U_j)$ an ordered spectral representation.

Let us prove first some additional mapping properties of $L$. Introduce the Sobolev space
\[
H^1(X^1,m^1):=\big\{
F\in \Hat H^1(X^1,m^1): \, \text{Eq. \eqref{eq-fcont} holds}
\big\}.
\]
Using the estimate \eqref{eq-sob2} and proceeding as in Lemma~\ref{lem4}
it is easy to see that $H^1(X^1,m^1)$
becomes a Hilbert space if equiped with the scalar product
\[
\langle F,G\rangle_{H^1(X^1,m^1)}:=\langle F',G'\rangle_{\ELL^2(X^1,m^1)}+
\langle F,G\rangle_{\ELL^2(X^1,m^1)}.
\] The following proposition is well known
for the locally finite graphs, see e.g.~\cite[Subsection~1.4.3]{BK}; we include the proof for the sake of completeness.
\begin{lemma}\label{lem6}
In $\ELL^2(X^1,m^1)$ consider a sesquilinear form $q$
defined on $H^1(X^1,m^1)$ by
$q(F,G)=\langle F',G'\rangle_{\ELL^2(X^1,m^1)}$. Then $q$ is closed and non-negative, and
$L$ is the self-adjoint operator associated with the form $q$.
\end{lemma}

\begin{proof}
The inequality $q(F,F)\ge 0$ is obvious, and the closedness is equivalent to the completeness of  $H^1(X^1,m^1)$.
Now let $T$ be the self-adjoint operator associated with $q$ and let $F\in \dom T$.
Take any $\varphi\in C^\infty_c(0,1)$ and any edge $[x,y]$ and consider a function $G$ on $X^1$ defined by
$G_{xy}:=\varphi$ and $G_{uv}:=0$ for $uv\notin\{xy,yx\}$. Clearly, $G$ belongs to $H^1(X^1,m^1)$, and we have
$q(F,G)=\langle TF, G\rangle_{\ELL^2(X^1,m^1)}$, which can be rewritten as
\[
\int_0^1 F'_{xy}(t) \overline{ \varphi'(t)}dt= \int_0^1 (TF)_{xy}(t)\overline{\varphi(t)}\,dt.
\]
As this holds for an arbitrary $\varphi\in C^\infty_c(0,1)$,
we have $F_{xy}\in H^2(0,1)$ and $(TF)_{xy}=-F''_{xy}$.
As $TF\in \ELL^2(X^1,m^1)$, we obtain the inclusion $\dom T\subset \Hat H^2(X^1,m^1)$,
and $F$ automatically satisfies the continuity condition~\eqref{eq-fcont} as it belongs to $H^1(X^1,m^1)$.

Now pick any $x\in X^0$ and denote by
$\ELL_x$ the set of the functions $G\in \Hat H^2(X^1,m^1)\cap H^1(X^1,m^1)$
such that $G_{uv}=0$ for $x\notin \{u,v\}$ and that $G(xy,1)=0$ for all
$y\sim x$. For $G\in \ELL_x$ we have again $\langle F',G'\rangle_{\ELL^2(X^1,m^1)}=-\langle F'',G\rangle_{\ELL^2(X^1,m^1)}$, which
 reads as
\[    
\sum_{y:y\sim x} c(xy)\int_0^1 F'_{xy}(t)\overline{ G'_{xy}(t)}\, dt=
-\sum_{y:y\sim x} c(xy)\int_0^1 F''_{xy}(t)\overline{ G_{xy}(t)}\, dt.
\]
Using the integration by parts we obtain
\begin{equation}
      \label{eq-fpg}
\sum_{y:y\sim x} c(xy)F'(xy,0) G(xy,0)=0.
\end{equation}
Now take $\varphi\in H^2(0,1)$ with $\varphi(0)=1$ and $\varphi(1)=0$ and define
a function $G$ on $X^1$ by $G_{xy}:=\varphi$ for $y\sim x$ and $G_{uv}:=0$ for $uv\notin\{xy,yx\}$,
then $G\in \ELL_x$ and $G(xy,0)=1$ for all $y\sim x$. Substituting this function $G$
into \eqref{eq-fpg} shows that $F$ satisfies the remaining boundary condition \eqref{eq-dder}.

The preceding constructions show that $T$ is a restriction of $L$. As both $T$ and $L$ are self-adjoint,
we have $T=L$.
\end{proof}

\begin{lemma}\label{prop4}
Let a function $w:X^1\to (0,+\infty)$ be such that $w(xy,t)\ge \sqrt{c(xy)}$
for $m^1$-a.e. $(xy,t)\in X^1$ and that $1/w\in \ELL^2(X^1,m^1)$.
Then there exists a Fourier-type expansion $(\Psi_{j})$ for $L$ with 
\begin{equation}
        \label{eq-hhh}
\cH_+=\ELL^2\Big(X^1,\dfrac{w^2}{c}\,m^1\Big), \quad
\cH_-=\ELL^2\Big(X^1,\dfrac{c}{w^2}\,m^1\Big).
\end{equation}
Here by $c:X^1\to\RR$ we mean the function given by $c(xy,t):=c(xy)$.
\end{lemma}

\begin{proof}
The operator $(1+L)^{-1/2}:\ELL^2(X^1,m^1) \to H^1(X^1,m^1)$ is bounded as $H^1(X^1,m^1)$ is the form domain of $L$ (see Lemma~\ref{lem6}).
On the other hand, the operator $B:H^1(X^1,m^1)\to\ELL^\infty (X^1,m^1)$ defined by
$BF(xy,t)=\sqrt{c(xy)}F(xy,t)$ is bounded due to \eqref{eq-sob2}.
Therefore, we see that the operator $B(1+L)^{-1/2}:\ELL^2(X^1,m^1)\to\ELL^\infty(X^1,m^1)$ is bounded.
Now we take the function $\gamma(t):=(1+|t|)^{-1/2}$ and consider the operator $T$ acting as
\[
TF(xy,t)=\dfrac{w(xy,t)}{\sqrt{c(xy)}}\, F(xy,t).
\]
The preceding discussion shows the operator
$T^{-1} \gamma(L)\equiv w^{-1} B(1+L)^{-1/2}$ is the composition of an operator defining a bounded map from $\ELL^2$ to $\ELL^\infty$
with a bounded operator acting from $\ELL^\infty$ to $\ELL^2$, and this composition is a Hilbert-Schmidt operator according
to the factorization principle. The adjoint operator $\gamma(L)T^{-1}$
is then also Hilbert-Schmidt, and Proposition~\ref{thm-gf} is applicable.
Due to the explicit choice of $T$, the associated spaces $\cH_\pm$ can be identified with the above
weighted $\ELL^2$-spaces.
\end{proof}

Now we construct explicitly a particular function $w$ satisfying the assumptions of Lemma~\ref{prop4};
its properties will be used in the subsequent argument. Assume that the condition \eqref{eq-m1}
is satisfied, denote
\[
\alpha(x):=\sum_{y:y\sim x} \sqrt{c(xy)}, \quad x\in X^0,
\]
and pick any family of positive numbers $v(x)$ with the following properties:
\begin{equation}
      \label{eq-vvx}
v(x)\ge \sqrt[4]{m^0(x)} \text{ for all } x\in X^0 \quad \text{ and } \quad \sum_{x\in X^0} \dfrac{\alpha(x)}{v(x)^2}<\infty,
\end{equation}
and for $(xy,t)\in X^1$ put
\begin{equation}
    \label{eq-ww}
w(xy,t)=\begin{cases}
v(x)\sqrt[4]{c(xy)},& t\in\big(0,\frac 12\big),\\
v(y)\sqrt[4]{c(xy)},& t\in\big(\frac 12,1\big).
\end{cases}
\end{equation}
We have clearly $w(xy,t)\ge \min\big(\sqrt[4]{m^0(x)},\sqrt[4]{m^0(y)}\,\big)\cdot\sqrt[4]{c(xy)}\ge \sqrt{c(xy)}$
and
\begin{align*}
\Big\|\dfrac{1}{w}\Big\|^2_{\ELL^2(X^1,m^1)}&=\dfrac{1}{2} \sum_{x\in X^0} \sum_{y:y\sim x} c(xy) \Big\|\dfrac{1}{w_{xy}^2}\Big\|^2_{\ELL^2(0,1)}\\
&=
\dfrac{1}{4} \sum_{x\in X^0} \sum_{y:y\sim x} \sqrt{c(xy)} \, \Big(
\dfrac{1}{v(x)^2}+\dfrac{1}{v(y)^2}\Big)\\
&=
\dfrac 1 2\sum_{x\in X^0} \sum_{y:y\sim x} \dfrac{\sqrt{c(xy)}}{v(x)^2}=\dfrac 1 2 \sum_{x\in X^0} \dfrac{\alpha(x)}{v(x)^2}
<\infty.
\end{align*}
Therefore, $1/w\in \ELL^2(X^1,m^1)$ and Lemma~\ref{lem4} is applicable. 

\begin{theorem} \label{th3}
Assume that the network satisfies \eqref{eq-m1}. Let $(\Psi_j)$ be the Fourier-type expansion for $L$ associated
with the function $w$ from \eqref{eq-ww}, then for $\mu$-a.e. $\lambda$ the associated functions $\Psi_{j,\lambda}$
can be represented as follows: there exist a function $b$ on $X^0$ and a function $a$ on $E$ such that
for $m^1$-a.e. $(xy,t)\in X^1$ the~associated generalized eigenfunctions $\Psi_{j,\lambda}$ admit the representation
\begin{equation}
        \label{eq-fbc}
\begin{gathered}
\Psi_{j,\lambda}(xy,t)=b(x)\cos (\sqrt\lambda t) +a(xy) \Phi(\sqrt \lambda,t),\\
\text{with }\sum_{y:y\sim x}c(xy)\big|a(xy)\big|<\infty \text{ and }\sum_{y:y\sim x}c(xy)a(xy)=0
\text { for all } x\in X^0.
\end{gathered}
\end{equation}
\end{theorem}

\begin{proof}
By Lemma~\ref{prop4}, $\Psi_{j,\lambda}\in \cH_-$ for $\mu$-a.e. $\lambda$.
We are going to show that the regularity properties of $\Psi_{j,\lambda}$
are in fact much better. Denote
\[
\dom_+ L:=\big\{ g\in \cH_+\cap\dom L: Lg\in\cH_+\big\}.
\]
Take an arbitrary $\varphi\in C_c^\infty(0,1)$. Pick $x,y\in X^0$ with $x\sim y$
and define a function $F$ on $X^1$ by $F_{xy}:=\varphi$ and
$F_{uv}:=0$ for all $uv\notin\{xy,yx\}$. Clearly, $F$ belongs to $\dom_+ L$, and
we have $\langle F'', \Psi_{j,\lambda}\rangle+\lambda \langle F, \Psi_{j,\lambda}\rangle=0$
according to the item (d) of Definition \ref{defin2}.
Due to the special structure of $F$ this can be rewritten as
\[
-\int_0^1\varphi''(t) \overline{\Psi_{j,\lambda}(xy,t)}\, dt=\lambda \int_0^1\varphi(t) \overline{\Psi_{j,\lambda}(xy,t)}\, dt.
\]
As $\varphi$ is arbitrary, this means that $t\mapsto \Psi_{j,\lambda}(xy,t)$ solves the equation $-u''+\lambda u=0$ in the space 
of distributions $\cD'(0,1)$.
Due to the ellipticity, it is automatically a $C^\infty$ solution 
and thus one can represent $\Psi_{j,\lambda}(xy,t)= a(xy)\Phi(\sqrt\lambda,t)+\beta(xy)\cos(\sqrt \lambda t)$
with some $a,\beta:E\to \CC$.
To show the representation \eqref{eq-fbc} we just need to show that $\Psi_{j,\lambda}$ satisfies the boundary conditions
\eqref{eq-fcont} and \eqref{eq-dder}. To see this, let us pick $x\in X^0$ and denote
$\Omega_x:=\big\{(xy,t):\, y\sim x,\, t\in\big[0,\frac 12\big)\big\}$.
Let $F$ be any function from $\dom_+ L$ vanishing outside $\Omega_x$, then, applying twice the integration by parts,
\begin{align}
0=\,&\,\langle -F'',\Psi_{j,\lambda}\rangle -\lambda \langle F,\Psi_{j,\lambda}\rangle\nonumber\\
=\,&-\sum_{y:y\sim x} c(xy)\bigg(
\int_0^1 F''(xy,t) \overline{\Psi_{j,\lambda}(xy,t)}\, dt\nonumber\\
&+\lambda \int_0^1 F(xy,t) \overline{\Psi_{j,\lambda}(xy,t)}\, dt
\bigg)\nonumber\\
=\,&
\sum_{y:y\sim x} c(xy)\bigg(
F'(xy,0)\overline{\Psi_{j,\lambda}(xy,0)}+
\int_0^1 F'(xy,t) \overline{\Psi'_{j,\lambda}(xy,t)}\, dt\nonumber\\
&-\lambda \int_0^1 F(xy,t) \overline{\Psi_{j,\lambda}(xy,t)}\, dt
\bigg)\nonumber\\
=\,&
\sum_{y:y\sim x} c(xy)\bigg(
F'(xy,0)\overline{\Psi_{j,\lambda}(xy,0)}-
F(xy,0)\overline{\Psi'_{j,\lambda}(xy,0)}\nonumber\\
&-\int_0^1 F(xy,t) \overline{\Psi''_{j,\lambda}(xy,t)}\, dt
-\lambda \int_0^1 F(xy,t) \overline{\Psi_{j,\lambda}(xy,t)}\, dt
\bigg)\nonumber\\
=\,&
\sum_{y:y\sim x} c(xy)\bigg(
F'(xy,0)\overline{\Psi_{j,\lambda}(xy,0)}-
F(xy,0)\overline{\Psi'_{j,\lambda}(xy,0)}\bigg).
      \label{eq-genf}
\end{align}
Choose a function $\psi\in H^2(0,1)$ such that $\psi(0)=0$, $\psi'(0)=1$
and $\psi(t)=0$ for $t\ge \frac 12$.
Note that if $x$ has just one neighbor, then the condition \eqref{eq-fcont}
is automatically satisfied. Assume now that $x$ has at least two neighbors
and let $u\sim x$, $y\sim x$ with $u\ne y$.
Construct a function $F$ on $X^1$ as follows: put $F_{xu}:=c(xy)\psi$,
$F_{xy}:=-c(xu)\psi$ and $F_{zv}:=0$ for $zv\notin\{xy,yx,xu,ux\}$.
Clearly, this $F$ belongs to $\dom L$ and $\cH_+$,
and $LF$ belongs to $\cH_+$ as well, so $F\in \dom_+ L$.
Substituting this function into \eqref{eq-genf}
we obtain the equality $c(xy)c(xu)\Psi_{j,\lambda}(xu,0)=c(xy)c(xu)\Psi_{j,\lambda}(xy,0)$,
which means that  $\Psi_{j,\lambda}(xu,0)=\Psi_{j,\lambda}(xy,0)$.
As $y$ and $u$ are arbitrary neighbors of $x$, we conclude that
$\Psi_{j,\lambda}$ satisfies the boundary conditions \eqref{eq-fcont}.
This means that there exist constants $b(x)$, $x\in X^0$, such that
$\beta(xy)=b(x)$ for all $y\sim x$, and we obtain
$\Psi_{j,\lambda}(xy,t)= b(x)\cos(\sqrt \lambda t)+a(xy)\Phi(\sqrt\lambda,t)$.

Denote by $\ELL^2(\Omega_x)$ the set of the measurable functions $G:\Omega_x\to \CC$ with
\[
\sum_{y:y\sim x} c(xy)^{3/2}\, \int_0^{1/2} \big|G(xy,t)\big|^2dt<\infty.
\] 
Note that $\Psi_{j,\lambda}$ belongs to $\ELL^2(\Omega_x)$; this follows from the estimates
\begin{align*}
\|\Psi_{j,\lambda}\|^2_- &=\dfrac{1}{2}\sum_{u\in X^0}\sum_{y:y \sim u}  \int_0^1 \big|\Psi_{j,\lambda}(uy,t)\big|^2 \dfrac{c(uy)^2}{w(uy,t)^2}dt\\
&\ge \sum_{y:y \sim x}  \int_0^{1/2} \big|\Psi_{j,\lambda}(xy,t)\big|^2 \dfrac{c(xy)^2}{w(xy,t)^2} dt\\
&= \dfrac{1}{v(x)^2} \sum_{y:y \sim x}  \int_0^{1/2} c(xy)^{3/2}\big|\Psi_{j,\lambda}(xy,t)\big|^2 dt.
\end{align*}
The function $G_1$ defined by $G_1(xy,t)=b(x)\cos (\sqrt\lambda t)$ belongs to $\ELL^2(\Omega_x)$ due to the equality
\begin{multline*}
\sum_{y:y\sim x} c(xy)^{3/2}\, \int_0^{1/2} \big|G_1(xy,t)\big|^2dt \\
= \big|b(x)\big|^2 \int_0^{1/2} \cos (\sqrt\lambda t)^2\, dt
\, \sum_{y:y\sim x} c(xy)^{3/2}<\infty.
\end{multline*}
Therefore, the function $G_2:=\Psi_{j,\lambda}-G_1$, i.e. $G_2(xy,t)=a(xy)\Phi(\sqrt\lambda,t)$, also belongs to $\ELL^2(\Omega_x)$, which gives
\[
\sum_{y:y\sim x} c(xy)^{3/2} \big|a(xy)|^2=\dfrac{\sum_{y:y\sim x} c(xy)^{3/2}
\big\| G_2(xy,\cdot)\big\|^2_{\ELL^2(0,1/2)}}{
\big\| \Phi (\sqrt\lambda , \cdot)\big\|^2_{\ELL^2(0,1/2)}}<\infty,
\]
and by the Cauchy-Schwarz inequality we have 
\[
\sum_{y:y\sim x} c(xy)\big|a(xy)\big|\le \Big(\sum_{y:y\sim x} c(xy)^{1/2}\Big)^{1/2}
\Big(\sum_{y:y\sim x} c(xy)^{3/2} \big|a(xy)\big|^2 \Big)^{1/2}<\infty.
\]
It remains to show that $\Psi_{j,\lambda}$
satisfies the boundary conditions \eqref{eq-dder}. Take a function $\psi\in H^2(0,1)$ such that
$\psi(0)=1$, $\psi'(0)=0$ and $\psi(t)=0$ for $t\ge \frac 12$, and
define a function $F$ on $X^1$ by $F_{xy}=\psi$ for $y\sim x$ and $F_{uv}:=0$ for $uv\notin\{xy,yx\}$.
For $k\in\{0,1,2\}$ one has
\[
\big\|F^{(k)}\big\|^2_{\ELL^2(X^1,m^1)}=\sum_{y:y\sim x} c(xy) \big\|F^{(k)}_{xy}\big\|^2_{\ELL^2(0,1)}
= m^0(x)\,\big\|\psi^{(k)}\big\|^2_{\ELL^2(0,1)},
\]
i.e. $F\in \Hat H^2(X^1,m^1)$, and $F$ clearly satisfies the boundary condition \eqref{eq-fcont},
and \eqref{eq-dder}, which means that $F\in \dom L$. On the other hand,
\begin{align*}
\|F^{(k)}\|^2_+&=\sum_{y:y\sim x} \int_0^1 \big|F^{(k)}(xy,t)\big|^2 w(xy,t)^2dt\\
&=
\sum_{y:y\sim x} \int_0^{1/2} \big|\psi^{(k)}(t)\big|^2 v(x)^2 \sqrt{c(xy)}dt\\
&= v(x)^2\int_0^{1/2} \big|\psi^{(k)}(t)\big|^2dt\cdot \sum_{y:y\sim x}  \sqrt{c(xy)}<+\infty,
\end{align*}
which shows that $F$ and $LF$ belong to $\cH_+$ and, therefore,  $F$ belongs to $\dom_+L$.
Substituting this function $F$ into \eqref{eq-genf} and using the equalities
$F(xy,0)=1$, $F'(xy,0)=0$ for $y\sim x$ we obtain
\[
\sum_{y:y\sim x} c(xy) \Psi_{j,\lambda}'(xy,0)=0.
\]
As $\Psi_{j,\lambda}'(xy,0)=a(xy)$, this completes the proof.
\end{proof}

\section{D'Alembert operators}\label{sec4}

\subsection{Main properties}

For $F \in \ELL^2(X^1,m^1)$ and $\tau\in\RR$ define $\Tilde F$
by the rules \eqref{eq-tilde}, \eqref{eq-tt1} and \eqref{eq-tt2},
and $C(\tau)F$ by the expression \eqref{eq-ctau}.
In general, the sums in the definition are infinite,
so we need to show that the above operations make sense.

\begin{lemma} \label{lem14}
For $F\in \ELL^2(X^1,m^1)$ and $\tau\in\RR$ define a function $F^\tau:X^1\to\CC$
by $F^\tau(xy,t)=\Tilde F(xy,\tau+t)$, then $F\mapsto F^\tau$ defines a bounded operator
in $\ELL^2(X^1,m^1)$.
So $C(\tau)$ is a bounded operator in $\ELL^2(X^1,m^1)$ for any $\tau\in\RR$.
\end{lemma}

\begin{proof}
Using the definition and the Cauchy-Schwarz inequality we have
\begin{align*}
\|F^{\tau+1}_{xy}\|^2_{\ELL^2(0,1)}&=
\bigg\| \dfrac{2}{m^0(y)} \sum_{v:v\sim y} c(yv) F^\tau_{yv}
-F^\tau_{yx}\bigg\|^2_{\ELL^2(0,1)}\\
&= \bigg\|  \sum_{v:v\sim y} \dfrac{c(yv)}{m^0(y)} \Big(2 F^\tau_{yv}-
F^\tau_{yx}\Big)\bigg\|^2_{\ELL^2(0,1)}\\
&\le
\Big( \sum_{v:v\sim y} \dfrac{c(yv)}{m^0(y)} \Big\|2 F^\tau_{yv}-F^\tau_{yx}\Big\|_{\ELL^2(0,1)}\Big)^2
\\
&\le
\bigg(
\sum_{v:v\sim y} \dfrac{c(yv)}{m^0(y)}
\bigg)\cdot
\bigg(
\sum_{v:v\sim y}
\dfrac{c(yv)}{m^0(y)} \cdot \Big\|2 F^\tau_{yv}-
F^\tau_{yx}\Big\|^2_{\ELL^2(0,1)}
\bigg)\\
&=
\sum_{v:v\sim y}
\dfrac{c(yv)}{m^0(y)} \cdot \big\|2 F^\tau_{yv}-
F^\tau_{yx}\big\|^2_{\ELL^2(0,1)}\\
&\le
8\sum_{v:v\sim y}
\dfrac{c(yv)}{m^0(y)} \,\big\|F^\tau_{yv}\big\|^2_{\ELL^2(0,1)}
+
2\big\| F^\tau_{yx}\big\|^2_{\ELL^2(0,1)}.
\end{align*}
Therefore,
\begin{multline*}
\big\|F^{\tau+1}\big\|^2_{\ELL^2(X^1,m^1)}=\dfrac{1}{2}\sum_{y\in X^0}\sum_{x:x\sim y}
c(xy)\big\|F^{\tau+1}_{xy}\big\|^2_{\ELL^2(0,1)}\\
\le
4 \sum_{y\in X^0}\sum_{x:x\sim y}
\sum_{v:v\sim y} c(xy)\dfrac{c(yv)}{m^0(y)} \,\big\|F^\tau_{yv}\big\|^2_{\ELL^2(0,1)}
+\sum_{y\in X^0}\sum_{x:x\sim y} c(xy)\big\| F^\tau_{yx}\big\|^2_{\ELL^2(0,1)}\\=10 \,\big\|F^{\tau}\big\|^2_{\ELL^2(X^1,m^1)}.
\end{multline*}
Similarly, using \eqref{eq-tt2} we show
\[
\|F^{\tau-1}\|^2_{\ELL^2(X^1,m^1)}\le 10 \,\|F^{\tau}\|^2_{\ELL^2(X^1,m^1)}
\]
Therefore, if $F\mapsto F^\tau$ is bounded, then $F\mapsto F^{\tau\pm 1}$
are bounded too. As $F\mapsto F^0$ is just the identity, we show by induction
that $F\mapsto F^\tau$ is bounded for all $\tau\in\ZZ$.
The boundedness for any $\tau$ follows from the majoration
\begin{multline*}
\big\| F^\tau_{xy}\big\|^2_{\ELL^2(0,1)}=\int_\tau^{\tau+1} \big|\Tilde F(xy,t)\big|^2dt\\
\le\int_n^{n+1} \big|\Tilde F(xy,t)\big|^2dt+\int_{n+1}^{n+2} \big|\Tilde F(xy,t)\big|^2dt\\
= \big\| F^n_{xy}\big\|^2_{\ELL^2(0,1)}
+\big\| F^{n+1}_{xy}\big\|^2_{\ELL^2(0,1)},
\end{multline*}
where $n\in\ZZ$ is chosen in such a way that $\tau\in [n,n+1)$. 
\end{proof}

\begin{lemma}
There holds 
\begin{align}
C(\tau+1)+C(\tau-1)&=2 C(1)C(\tau) &\text{for } &\tau\in\RR,  \label{eq-cc2}\\
\text{and} \quad
     \label{eq-lem20}
 C(2\tau)+1&=2C(\tau)^2 &\text{ for }&\tau\in\Big[\,0,\dfrac 12\Big].
\end{align}
\end{lemma}

\begin{proof}
Take any $F\in \ELL^2(X^1,m^1)$ and denote $K:=C(\tau) F$, then
for any $t\in(0,1)$ we have:
\begin{multline}
C(1)C(\tau) F(xy,t)= \dfrac{1}{2} \Big( \Tilde K(xy,t+1)+\Tilde K(xy,t-1)\Big)\\
\begin{aligned}
=&\dfrac{1}{m^0(y)} \sum_{v:v\sim y} c(yv) K(yv,t)
+ \dfrac{1}{m^0(x)} \sum_{u:u\sim x} c(ux) \Tilde K(ux,t)
-\Tilde K(yx,t)\\
=&\dfrac{1}{ 2 m^0(y)} \sum_{v:v\sim y} c(yv) \big( \Tilde F(yv,t+\tau)+\Tilde F(yv,t-\tau)\big)\\
&+ \dfrac{1}{2 m^0(x)} \sum_{u:u\sim x} c(ux) \big( \Tilde F(ux,t+\tau)+\Tilde F(ux,t-\tau)\big)\\
&-\dfrac{1}{2}\,\Big(
\Tilde F(yx,t+\tau)+\Tilde F(yx,t-\tau)
\Big).
\end{aligned}
  \label{eq-cct}
\end{multline}
Now we have, using \eqref{eq-tt1} and \eqref{eq-tt2},
\begin{align*}
C(\tau+1)F(xy,t)=&\dfrac{1}{2}\,\Big(\Tilde F(xy,t+\tau+1)+\Tilde F(xy,t-\tau-1)\Big)\\
=&\dfrac{1}{m^0(y)} \sum_{v:v\sim y} c(yv) \Tilde F(yv,t+\tau) -\dfrac{1}{2}\,\Tilde F(yx,t+\tau),\\
&+
\dfrac{1}{m^0(x)} \sum_{u:u\sim x} c(ux) \Tilde F(ux,t-\tau) - \dfrac 1 2  \,\Tilde F(yx,t-\tau)\\
C(\tau-1)F(xy,t)=&\dfrac{1}{2}\,\Big(\Tilde F(xy,t+\tau-1)+\Tilde F(xy,t-\tau+1)\Big)\\
=&\dfrac{1}{m^0(x)} \sum_{u:u\sim x} c(ux) \Tilde F(ux,t+\tau ) - \dfrac 1 2 \,\Tilde F(yx,t+\tau)\\
 &+\dfrac{1}{m^0(y)} \sum_{v:v\sim y} c(yv) \Tilde F(yv,t-\tau ) - \dfrac 1 2 \,\Tilde F(yx,t-\tau).
\end{align*}
Comparing with  \eqref{eq-cct} we get the identity \eqref{eq-cc2}.

The proof of \eqref{eq-lem20} follows the same scheme. Denote $G:=C(\tau)F$, then
\[
G(xy,t)
=\begin{cases}
\dfrac 12\, F(xy,t+\tau)-\dfrac12 \,F(yx,t-\tau+1)\\
\qquad+\dfrac{1}{m^0(x)} \sum\limits_{u:u\sim x} c(ux) F(ux,t-\tau+1) \text{ for }0<t< \tau,\\[\bigskipamount]
\dfrac 12\, F(xy,t+\tau)+\dfrac12 \,F(xy,t-\tau) \quad  \text{ for }\tau\le t\le 1-\tau,\\[\bigskipamount]
\dfrac 12\, F(xy,t-\tau)-\dfrac12 \,F(yx,t+\tau-1)\\
\qquad + \dfrac{1}{m^0(y)}\sum\limits_{v:v\sim y} c(yv) F(yv,t+\tau-1) \text{ for }1-\tau<t<1.
\end{cases}
\]
After an elementary algebra we obtain
\begin{align*}
\Tilde G(xy,t+\tau)
&=\begin{cases}
\dfrac 12 \,F(xy,t+2\tau)+\dfrac 12 \,F(xy,t) \text{ for } 0< t \le 1-2\tau,\\[\bigskipamount]
\dfrac{1}{m^0(y)}\sum\limits_{v:v\sim y} c(yv) F(yv,t+2\tau-1)\\[\bigskipamount]
\quad-\dfrac12\,F(yx,t+2\tau-1)+\dfrac{1}{2} \,F(xy,t) \text{ for } 1-2\tau<t< 1
\end{cases}\\[\medskipamount]
&=\dfrac12\, \Tilde F(xy,t+2\tau)+\dfrac12\, F(xy,t),
\end{align*}
and
\begin{align*}
\Tilde G(xy,t-\tau)
&=\begin{cases}
\dfrac{1}{m^0(x)}\, \sum\limits_{u:u\sim x} c(ux) F(ux,t-2\tau+1)-\dfrac12\, F(yx,t-2\tau+1)\\[\bigskipamount]
\quad+\dfrac12\, F(xy,t) \text{ for } 0<t\le 2\tau,\\[\bigskipamount]
\dfrac12\, F(xy,t)+\dfrac12\, F(xy,t-2\tau) \text{ for } 2\tau<t< 1
\end{cases}\\[\medskipamount]
&=\dfrac12\, \Tilde F(xy,t-2\tau)+\dfrac12\, F(xy,t),
\end{align*}
which gives the sought result.
\end{proof}

\subsection{Relation between $L$ and $C(\tau)$}

We now use the above preparations and the Fourier-type expansion from section~\ref{sec3}
to prove Theorem~\ref{th4}.

\begin{proof}[Proof of Theorem \ref{th4}]
Note first that it is sufficient to show the result for $\tau\in[0,\frac 12]$.
Indeed, it extends to the values $\tau\in[0,1]$ using Eq.~\eqref{eq-lem20} and the identity
$\cos (\alpha \tau)= 2 \cos^2(\alpha\tau/2)-1$. Then we extend the equality the the values
$\tau\in[-1,1]$  by the symmetry:
$C(\tau)=C(-\tau)$, $\cos(\tau\sqrt L)=\cos(-\tau\sqrt L)$,
and, finally, we prove it for all values of $\tau$ by induction using \eqref{eq-cc2} and
the identities $\cos( \alpha (\tau\pm1)\big)=2\cos (\alpha)\cos (\alpha\tau)-\cos( \alpha (\tau\mp1)\big)$.
So from now on we assume that $\tau\in[0,\frac 12]$.

As both operators $C(\tau)$ and $\cos(\tau\sqrt L)$ are bounded, it is sufficient to show
that they coincide on a dense subspace.
Let $\varphi\in C_c^\infty(0,1)$ with $\supp \varphi\subset (0,\frac 12)$.
Pick $x,y\in X$ with $x\sim y$ and denote by $F\equiv \Pi(\varphi, xy)$
the function on $X^1$ with $F_{xy}=\varphi$ and
$F_{uv}=0$ for $uv\notin\{xy,yx\}$. As $xy$ and $\varphi$
are arbitrary, the subspace spanned by the functions $\Pi(\varphi,xy)$
is dense in $\ELL^2(X^1,m^1)$.
Therefore, it is sufficient to show that $C(\tau)F=\cos(\tau\sqrt L)F$
for the above function $F$.

Clearly, $F\in\cH_+$. Let us show that the function $G:=C(\tau)F$ belongs to $\cH_+$ too.
We have by direct computation,
\begin{align*}
\text{for }u\sim x, \, u\ne y: \quad&&
\Tilde F(xu,t+\tau)&=0,\\
&& \Tilde F(xu,t-\tau)&=\begin{cases}
\dfrac{2 c(xy)}{m^0(x)}\, \varphi(\tau-t), & t\in [0,\tau),\\
0, & t\in[\tau,1],
\end{cases}\\
\text{and} \qquad
&& \Tilde F(xy,t+\tau)&=\begin{cases}
\varphi(t+\tau), & t\in[0,1-\tau],\\[\bigskipamount]
0,& t\in(1-\tau,1],
\end{cases}\\
&& \Tilde F(xy,t-\tau)&=\begin{cases}
\Big(\dfrac{2 c(xy)}{m^0(x)}-1\Big) \varphi(\tau-t), & t\in[0,\tau),\\[\bigskipamount]
\varphi(t-\tau), & t\in[\tau,1],
\end{cases}
\end{align*}
and $\Tilde F (uv,t\pm\tau)=0$ for $x\notin\{u,v\}$ and $t\in[0,1]$.
This gives
\begin{align*}
G(xu,t)&=\begin{cases}
\dfrac{c(xy)}{m^0(x)}\, \varphi(\tau-t), & t \in[0,\tau],\\[\bigskipamount]
0, & t\in (\tau,1],
\end{cases} \quad \text{ for } u\sim x, \, u\ne y,\\[\bigskipamount]
G(xy,t)&=\begin{cases}
\dfrac{c(xy)}{m^0(x)}\, \varphi(\tau-t) -\dfrac{1}{2}\, \varphi(\tau-t) +\dfrac{1}{2}\, \varphi(t+\tau), \quad t\in[0,\tau],\\[\bigskipamount]
\dfrac{1}{2}\, \varphi(t+\tau)+\dfrac{1}{2} \,\varphi(t-\tau), \quad t\in(\tau,1-\tau),\\[\bigskipamount]
\dfrac{1}{2}\, \varphi(t-\tau), \quad t\in[1-\tau,1],
\end{cases}
\end{align*}
and $G_{uv}=0$ for $x\notin\{u,v\}$.
Denote $\kappa:=\|\varphi\|_\infty+\|\varphi'\|_\infty+\|\varphi''\|_\infty$,
then for any $k\in\{0,1,2\}$ and all $u\sim v$ one has
$\|G^{(k)}_{uv}\|_\infty\le 2\kappa$.
The inclusion $G\in \cH_+$ means that
\[
\sum_{u\in X^0}\sum_{v:v\sim u} \int_0^1 \big|G(uv,t)\big|^2 w(uv,t)^2\,dt<+\infty.
\]
In our case,
\begin{multline}
   \label{eq-hplus}
\dfrac 12 \sum_{u\in X^0}\sum_{v:v\sim u} \int_0^1 \big|G(uv,t)\big|^2 w(uv,t)^2\,dt\\
=
\int_{\tau}^1 \big|G(xy,t)\big|^2 w(xy,t)^2dt
+
\sum_{u:u\sim x} \int_0^\tau\big|G(xu,t)\big|^2 w(xu,t)^2dt\\
\le 4\kappa^2 \int_{\tau}^1  w(xy,t)^2dt
+ 4\kappa^2\,\sum_{u:u\sim x} \int_0^\tau w(xu,t)^2dt
\\
\le 4\kappa^2\big(v(x)^2+v(y)^2\big) + 4\kappa^2 v(x)^2 \sum_{u:u\sim x} \sqrt{c(xu)}<\infty.
\end{multline}
Therefore, $G=C(\tau)F \in\cH_+$, and for $\mu$-a.e. $\lambda$ we have, with
$\Psi:=\Psi_{j,\lambda}$,
\begin{multline*}
\big(U_j C(\tau)F\big)(\lambda)=\big\langle C(\tau)F,  \Psi_{j,\lambda} \big\rangle\\
=\dfrac{1}{2}\, \sum_{u\in X^0}\sum_{v:v\sim u} c(uv)
\int_0^1 \big(C(\tau)F \big)(uv,t)\, \overline{\Psi(uv,t)}\, dt=I_1+I_2,
\end{multline*}
where
\begin{align*}
I_1&=\sum_{u:u\sim x} c(xu) \,\dfrac{c(xy)}{m^0(x)} \int_0^\tau
\varphi(\tau-t) \overline{\Psi(xu,t)}\, dt,\\
I_2&=\dfrac{c(xy)}{2}\,\Big(
{}-\,\int_0^\tau \varphi(\tau-t) \overline{\Psi(xy,t)}dt\\
&\quad+\int_\tau^1 \varphi(t-\tau)\overline{\Psi(xy,t)}dt+\int_0^{1-\tau}\varphi(t+\tau)\overline{\Psi(xy,t)}dt
\Big).
\end{align*}
Using Theorem~\ref{th3} one can represent
$\Psi(xu,t)=B \cos(\sqrt \lambda t)+A(u) \Phi(\sqrt\lambda,t)$
with some constants $B$ and $A(u)$, $u\sim x$, such that
\[
\sum_{u:u\in x} c(xu) \big|A(u)\big|<+\infty \text{ and} \sum_{u:u\in x} c(xu) A(u)=0.
\]
Using this representation we obtain
\begin{align*}
I_1=&\,\overline B \dfrac{c(xy)}{m^0(x)} \sum_{u:u\sim x} c(xu) \int_0^\tau \varphi(t) \cos\big(\sqrt \lambda (t-\tau)\big)dt\\
&-\dfrac{c(xy)}{m^0(x)} \sum_{u:u\sim x} c(xu) \overline{A(u)} \int_0^\tau \varphi(t)\Phi(\sqrt\lambda,t-\tau)dt\\
=&\,\overline Bc(xy)\int_0^\tau \varphi(t) \cos\big(\sqrt \lambda (t-\tau)\big)dt\\
=&\,\overline Bc(xy) \Big(
\cos(\sqrt\lambda \tau)\int_0^\tau \varphi(t) \cos(\sqrt \lambda t)dt+
\sin(\sqrt\lambda \tau)\int_0^\tau \varphi(t) \sin(\sqrt \lambda t)dt
\Big)
\end{align*}
and
\begin{align*}
\dfrac{2I_2}{c(xy)}=&-\int_0^\tau \varphi(t) \overline{\Psi(xy,\tau-t)}dt
+\int_0^{1-\tau} \varphi(t)\overline{\Psi(xy,t+\tau)}dt\\
&\quad+\int_\tau^1\varphi(t)\overline{\Psi(xy,t-\tau)}dt\\
=&\,\overline B K_1 +\overline A(y) K_2
\end{align*}
with
\begin{align*}
K_1:=& -\int_0^\tau \varphi(t) \cos\big(\sqrt \lambda (t-\tau)\big)dt
+ \int_0^{1-\tau} \varphi(t) \cos\big(\sqrt \lambda (t+\tau)\big)dt\\
&\quad+\int_\tau^1 \varphi(t) \cos\big(\sqrt \lambda (t-\tau)\big)dt,\\
=&-\cos(\sqrt \lambda \tau)\int_0^\tau \varphi(t) \cos(\sqrt \lambda t)dt-\sin(\sqrt \lambda \tau)\int_0^\tau \varphi(t) \sin(\sqrt \lambda t)dt\\
&+\cos(\sqrt \lambda \tau)\int_0^1 \varphi(t) \cos(\sqrt \lambda t)dt-\sin(\sqrt \lambda \tau)\int_0^1 \varphi(t) \sin(\sqrt \lambda t)dt\\
&+\cos(\sqrt \lambda \tau)\int_\tau^1 \varphi(t) \cos(\sqrt \lambda t)dt+\sin(\sqrt \lambda \tau)\int_\tau^1 \varphi(t) \sin(\sqrt \lambda t)dt\\
=&\,2\cos(\sqrt \lambda \tau)\int_\tau^1 \varphi(t) \cos(\sqrt \lambda t)dt
-2\sin(\sqrt \lambda \tau)\int_0^\tau \varphi(t) \sin(\sqrt \lambda t)dt
\\
K_2:=&\int_0^\tau \varphi(t) \Phi(\sqrt \lambda,t-\tau)dt
+ \int_0^{1-\tau} \varphi(t) \Phi(\sqrt \lambda, t+\tau)dt\\
&\quad +  \int_\tau^1 \varphi(t) \Phi(\sqrt \lambda,t-\tau)dt\\
=&\int_0^1 \varphi(t) \Big(\Phi(\sqrt \lambda, t+\tau)+\Phi(\sqrt \lambda, t-\tau)\Big)dt\\
=&\,2\cos (\sqrt\lambda \tau) \int_0^1 \varphi(t) \Phi(\sqrt \lambda, t)dt;
\end{align*}
above we used the fact that the integration on $[0,1-\tau]$ is equivalent to the integration on $[0,1]$
because $\varphi$ vanishes in $[1-\tau,1]$.
Finally we arrive at
\begin{align*}
\big(U_j C(\tau)F\big)(\lambda)=\,&
\big\langle C(\tau)F,  \Psi_{j,\lambda} \big\rangle=I_1+I_2\\
=\, &\overline Bc(xy) 
\cos(\sqrt\lambda \tau)\int_0^\tau \varphi(t) \cos(\sqrt \lambda t)dt\\
&+\overline Bc(xy) \sin(\sqrt\lambda \tau)\int_0^\tau \varphi(t) \sin(\sqrt \lambda t)dt\\
&+\overline Bc(xy)\cos(\sqrt \lambda \tau)\int_\tau^1 \varphi(t) \cos(\sqrt \lambda t)dt\\
&-\overline Bc(xy)\sin(\sqrt \lambda \tau)\int_0^\tau \varphi(t) \sin(\sqrt \lambda t)dt\\
&+\overline{A(y)}\, c(xy)\cos (\sqrt\lambda \tau) \int_0^1 \varphi(t) \Phi(\sqrt \lambda, t)dt\\
=\, &c(xy)\cos(\tau\sqrt\lambda )\int_0^1\varphi(t) \overline{\big(B \cos (\sqrt\lambda \,t) + A(y)\Phi(\sqrt \lambda, t)\big)}dt\\
=\, &c(xy)\cos(\tau\sqrt\lambda)\int_0^1 \varphi(t)\overline{\Psi(xy,t)}dt\\
=\, &\cos(\tau \sqrt\lambda )\big\langle F,  \Psi_{j,\lambda} \big\rangle=\cos(\sqrt\lambda \,\tau)\big(U_j F\big)(\lambda).
\end{align*}
Therefore, $U C(\tau) U^*$ coincides on a dense subspace
with the operator of multiplication with the function $\lambda\mapsto \cos(\tau\sqrt{\lambda})$.
On the other hand, $U \cos(\tau\sqrt L)U^*$ is the same multiplication operator due
to the definition of an ordered spectral representation (Definition \ref{defin2}). This shows
the equality $C(\tau)=\cos(\tau\sqrt L)$.
\end{proof}

\subsection{Relation between $L$ and $A$}\label{swave}

For $t\in \RR$ define a function
\[
S_t(z):=\Phi(z,t)=\begin{cases}
t, & z=0,\\[\medskipamount]
\dfrac{\sin (z t)}{z}, & z\ne 0.
\end{cases}
\]
Let us recall the following well known fact, see e.g. Proposition~6.7 in \cite{ksm}:
\begin{prop}
Let $T$ be a self-adjoint operator in a Hilbert space $\cH$ and $T\ge 0$.
For $F_0\in \dom T$ and $F_1\in\dom\sqrt T$ define a function
$G:\RR\to \cH$ by 
$G(t):=\cos\big( t \sqrt T\,\big) F_0+ S_t\big(\sqrt T\big) F_1$,
then $G$ is the unique $C^2(\RR,\cH)$-solution of the initial
value problem for the wave equation associated with $T$,
\[
G''+TG=0, \quad G(0)=F_0,\quad G'(0)=F_1.
\]
Furthermore, for any $F\in\dom\sqrt T$ and any $\tau\in\RR$ we have
\begin{equation}
     \label{eq-sint}
S_\tau(\sqrt T) F=\int_0^\tau \cos \big(\sigma\sqrt T\big) F\, d\sigma.
\end{equation}
\end{prop}
We remark that for $\ker T=\{0\}$ in the above theorem one has simply
$S_t (\sqrt T)=\big(\sqrt T\big)^{-1}\sin\big(t\sqrt T)$.

Now let us prove Theorem~\ref{th6}.

\begin{proof}[Proof of Theorem \ref{th6}]
 Actually we are going to show that $A=S_1(\sqrt L)\equiv \Phi(\sqrt L)$. As both $A$
 and $\Phi(\sqrt L)$ are bounded, it is sufficient to show that they coincide the dense subspace $\dom\sqrt L$.
Take an arbitrary $F\in \dom\sqrt L\equiv H^1(X^1,m^1)$, then, by combining Theorem~\ref{th4}, Lemma~\ref{lem6}
and the identity~\eqref{eq-sint}, for $m^1$-a.e. $(xy,t)\in X^1$ we have
\[
G(xy,t):=S_1(\sqrt L)F(xy,t)=\int_0^1 C(\tau)F(xy,t)\, d\tau.
\]
Let $t\in[0,\frac 12]$, then
\[
C(\tau)F(xy,t)=\begin{cases}
\dfrac12\, F(xy,t-\tau)+\dfrac12\, F(xy,t+\tau), & \tau \in(0,t),\\[\bigskipamount]
\dfrac12\, F(xy,t+\tau)-\dfrac12\, F(yx,t-\tau+1)\\[\medskipamount]
\qquad+\dfrac{1}{m^0(x)}\,
\sum\limits_{u:u\sim x} c(ux) F(ux,t-\tau+1), & \tau\in(t,1-t),\\[\bigskipamount]
-\dfrac12\, F(yx,t-\tau+1)-\dfrac12\, F(yx,t+\tau-1)\\
\quad +\dfrac{1}{m^0(x)}\,
\sum\limits_{u:u\sim x} c(ux)F(ux,t-\tau+1),\\
\quad +\dfrac{1}{m^0(y)}\,
\sum\limits_{v:v\sim y} c(yv)F(yv,t+\tau-1),& \tau\in (1-t,1),
\end{cases}
\]
which gives
\begin{align*}
G(xy,t)&=\int_0^t C(\tau)F(xy,t)d\tau+\int_t^{1-t} C(\tau)F(xy,t)d\tau\\
&\quad +\int_{1-t}^1 C(\tau)F(xy,t)d\tau\\
&=\dfrac12\,\bigg(\int_0^t F(xy,t-\tau)d\tau+\int_0^{1-t} F(xy,t+\tau)d\tau\\
&\qquad-\int_t^1 F(yx,t-\tau+1)d\tau -\int_{1-t}^1 F(yx,t+\tau-1)d\tau\bigg)\\
&\qquad+\dfrac{1}{m^0(x)}\, \sum\limits_{u:u\sim x}c(ux)\int_t^1 F(ux,t-\tau+1)d\tau\\
&\qquad+\dfrac{1}{m^0(y)}\, \sum\limits_{v:v\sim y}c(yv)\int_{1-t}^1 F(yv,t+\tau-1)d\tau.
\end{align*}
The sum of the first four terms equals zero due to
\begin{multline*}
\int_0^t F(xy,t-\tau)d\tau+\int_0^{1-t} F(xy,t+\tau)d\tau\\
{}-\int_t^1 F(yx,t-\tau+1)d\tau
-\int_{1-t}^1 F(yx,t+\tau-1)d\tau\\
=\int_0^t F(xy,\tau)d\tau+\int_t^1 F(xy,\tau)d\tau
-\int_t^1 F(yx,\tau)d\tau-\int_0^t F(yx,\tau)d\tau\\
=\int_0^1 F(xy,\tau)d\tau-
\int_0^1 F(yx,\tau)d\tau=0,
\end{multline*}
and we arrive at
\begin{align*}
G(xy,t)=&\,\dfrac{1}{m^0(x)}\, \sum\limits_{u:u\sim x}c(ux)\int_t^1 F(ux,t-\tau+1)d\tau\\
&\quad+\dfrac{1}{m^0(y)}\, \sum\limits_{v:v\sim y}c(yv)\int_{1-t}^1 F(yv,t+\tau-1)d\tau\\
=&\,\dfrac{1}{m^0(x)}\, \sum\limits_{u:u\sim x}c(xu)\int_0^{1-t} F(xu,\tau)d\tau\\
&\quad+\dfrac{1}{m^0(y)}\, \sum\limits_{v:v\sim y}c(yv)\int_0^t F(yv,\tau)d\tau=AF(xy,t).
\end{align*}
For $t\in[\frac12,1]$
we have $1-t\in [0,\frac 12]$, and by the preceding computation we have $G(xy,t)=G(yx,1-t)=AF(yx,1-t)=AF(xy,t)$.
\end{proof}

It is worth noting that a certain representation of the solutions to the wave equation
is available for finite networks with non-constant edge lengths too~\cite{Kor-Pryad,dan},
and it would be interesting to understand whether an analog of the averaging operator
can be calculated in a more or less explicit way.

\section*{Acknowledgments}

The first named author gratefully acknowledges the financial support from the German Research Foundation (DFG).
The second named author was supported by ANR NOSEVOL and by GDR DYNQUA.
A large of part of this work was done during the mini-workshop ``Boundary value problems and spectral geometry'' 
in winter 2012 at the Mathematical Institute in Oberwolfach~\cite{ober}, during the visit
of the second named author  at the Mathematical Institute of the Friedrich Schiller University in Jena in summer 2012,
and during the stay of the both authors at the Camille Jordan Institute in Lyon in spring 2013,
and we express our gratitude to these institutions
for the warm hospitality and the facilities provided.

\end{document}